\documentclass[10pt]{article}
\usepackage{amsmath,amssymb,amsthm,amscd}
\numberwithin{equation}{section}

\DeclareMathOperator{\ev}{ev} 

\newcommand{\bc}{\begin{center}}
\newcommand{\ec}{\end{center}}
\newcommand{\ba}{\begin{array}}
\newcommand{\ea}{\end{array}}

\renewcommand{\dfrac}{\displaystyle\frac }

\newtheorem{thm}{Theorem\ }[section]

\newtheorem{rem}{ Remark\ }[section]
\newtheorem{cor}{Corollary }[section]

\newtheorem{lem}{ Lemma\ }[section]

\numberwithin{equation}{section}
\numberwithin{thm}{section}
\numberwithin{defn}{section}
\numberwithin{lem}{section}
\numberwithin{cor}{section}
\numberwithin{rem}{section}

\def\cM{{\mathcal M}}

\def\bF{{\mathbf F}}

\def\bF{{\mathbf F}}

\def\ev{{\mathbf e}{\mathbf v}}

\def\ZZ{{\mathbb Z}}
\def\QQ{{\mathbb Q}}
\def\RR{{\mathbb R}}
\def\CC{{\mathbb C}}

\def\begeq{\begin{equation}}
\def\endeq{\end{equation}}

\title{Asypmtotics of enumerative invariants in $\CC P^2$}
\author{Gang Tian\thanks{Supported partially by
grants from NSF and NSFC}
\\Peking University and Princeton University\\
Dongyi Wei    \\
School of Math Sciences and BICMR, Peking University.}

\date{}

\begin{document}

\maketitle

\tableofcontents
\begin{abstract}
In this paper, we give the asymptotic expansion of $n_{0,d}$ and $n_{1,d}$.
\end{abstract}
\section{Introduction}
The GW-invariants were constructed first by Ruan-Tian for semi-positive symplectic manifolds (see \cite{ruantian1}, \cite{ruantian2}, \cite{mcduffsalamon}) and subsequently for general symplectic manifolds (see \cite{litian}, \cite{fukayaono} etc.). For any compact symplectic manifold $M$ of dimension $2n$, its GW-invariants are given by a family of the following multi-linear maps:
\begin{equation}\label{eq:gw-1}
\Phi^M_{g,a,k}: H^*(\cM_{g,k},\QQ)\times H^*(M, \QQ)^k\,\mapsto \, \QQ,
\end{equation}
where $a\in H_2(M,\ZZ)$ and $\cM_{g,k}$ denotes the moduli of stable curves of genus $g$ and with $k$ marked points. For any $\beta\in H^*(\cM_{g,k}, \QQ)$
and $\alpha_1,\cdots, \alpha_k\in H^*(M,\QQ)$, we have
$$\Phi^M_{g,a,k}(\beta;\alpha_1,\cdots, \alpha_k)\,=\,\int_{\cM^{vir}(g,a,k)}\,\ev^*(\beta\wedge\alpha_1\wedge\cdots\wedge \alpha_k) ,$$
where $\cM^{vir}(g,a,k)$ is the virtual moduli of stable maps of genus $g$, homology class $a$ and $k$ marked points and $\ev$ is the evaluation map from stable maps
to $\cM_{g,k}\times M^k$.
As usual, we consider the generating function
\begin{equation}\label{eq:gw-2}
\bF_g^M(w)\,=\,\sum_{a, k}\,\frac{\Phi^M_{g,a,k} (1; w,\cdots,w)}{k!},~~~w\,\in\,H^*(M,\CC).
\end{equation}
For simplicity, we restrict $\bF^M_g$ to the even part $H_{ev}(M)$ of $H^*(M,\CC)$. Let $\gamma_0, \cdots, \gamma_L$ be an integral basis for $H_{ev}(M)$ such that each $\gamma_i$ is of pure degree $2 d_i$ and $0=d_0<1=d_1\le \cdots\le d_L=n$. Write $w = \sum t_i \gamma_i$, then the restriction of $\bF_g$ is of the form
$$\bF_g^M(t_0,\cdots, t_L)\,=\, \sum_{a,k} \sum_{k_0+\cdots+k_L = k}\,\psi_{g,a}(k_0,\cdots, k_L) \,t_1^{k_1}\cdots t_L^{k_L},$$
where $\psi_{g,a}(k_0,\cdots,k_L)$ denotes the evaluation of $\Phi^M_{g,a,k}$ at $k_0 $ of $\gamma_0$, ......, $k_L$ of $\gamma_L$ divided by $k_0!\cdots k_L!$.
We observe that $\psi_{g,a}(k_0,\cdots,k_L)=0$ unless
\begin{equation}\label{eq:gw-3}
c_1(M)(a) \,+\, (3-n)(g-1)\, = \, \sum_{i=0}^L \,k_i\,(d_i - 1).
\end{equation}
We further observe that $\psi_{g,a}(k_0,\cdots,k_L)=0$ whenever $k_0 >0$ and $a\not= 0$ and
$$\psi_{g,a}(0,k_1,\cdots,k_{m-1},k_m,\cdots,k_L)\,=\,\left(\prod_{i=1}^{m-1} \,(\gamma_i(a))^{k_i}\right)\,\psi_{g,a}(0,\cdots,0,k_m,\cdots, k_L),$$
where $m\ge 2$ is chosen such that $d_i=1$ for $1\le i\le m-1$ and $d_i \ge 2$ for any $i\ge m$.

The problem we concern is on the asymptotic of $\psi_{g,a}$ as $ c_1(M)(a)\mapsto +\infty$.
More precisely, let $k_{m}(\ell), \cdots, k_L(\ell)$ satisfy:
\eqref{eq:gw-3} holds with $a$ replaced by $\ell a$ and for any $N > 0$ and $i\ge m$, $k_i(\ell)=\ell \sum_{j=0}^N c_{ij} \ell^{-j} + o(\ell^{-N})$ as $\ell \to \infty$. We expect that there are $d$, $c$ and $b_0,b_1,\cdots$ such that
\begin{equation}\label{eq:gw-4}
\psi_{g,\ell a}(0,0,\cdots,0,k_m(\ell),\cdots,k_L(\ell))\,=\, \ell^d\,e^{-c\,\ell}\,\sum_{i=0}^{N}\,( b_i \, \ell^{-i}\,+\,o(\ell^{-N})).
\end{equation}
A deeper problem is what geometric information is encoded in those coefficients $b_0,b_1,\cdots$, in other words, can one express those $b_i$ in terms of
geometric quantities of $M$.

In this short paper, we study the first problem above on asymptotic in the case of $M=\CC P^2$ and $g=0,1$. Since $H^0(\CC P^2, \ZZ)=H^2(\CC P^2, \ZZ)=H^4(\CC P^2, \ZZ)=\ZZ$,
we can denote $a$ by $d\in \ZZ$ and have only one cohomology class $\gamma_2$ of degree greater than $2$. Thus it suffices to consider
$\psi_{g,d}(0,0,3d-1)$ in the above problem. For simplicity, we denote it by $n_{g,d}$. Note that
$$(3d-1)!\, n_{0,d}\,=\,\psi_{0,d} (0,0,3d-1).$$
It counts the number of rational curves of degree $d$ in $\CC P^2$ through $3d-1$ points in general position.

Our first main result is
\begin{thm}
\label{th:main}
In the case of genus $g=0$, we have the following asymptotic expansion: There exist $ x_0,\ a_k^0\in\RR$ such that $a_3^0>0$ and for any $N \ge 4$,
$$n_{0,d}\,=\,e^{-dx_0}\,\left(\sum_{ k=3}^{N-1}\,a_{k}^0\,d^{-k-1/2}\,+\,O(d^{-N-1/2})\right).$$
\end{thm}

\begin{cor}\label{cor:1}
Let $n_{0,d}$ be the enumerative invariants as above for $\CC P^2$, then
\begin{equation}\label{eq:gw-4-1}
\lim_{d\to \infty}\,\sqrt[d]{n_{0,d}}\,=\,e^{-x_0}.
\end{equation}
\end{cor}
In \cite{francesco}, \eqref{eq:gw-4-1} was claimed by P. Di Francesco and C. Itzykson without a valid proof. In \cite{zinger11}, A. Zinger verified some claims in \cite{francesco} and proposed a proof of \eqref{eq:gw-4-1} under a conjectured condition which is still open.

Our second result is
\begin{thm}
\label{th:main-2}
In the case of genus $g=1$, we have the following asymptotic expansion: There exists $ a_k^1\in\RR$ and $x_0$ same as in Theorem \ref{th:main} such that for any $N \ge 4$,
$$n_{1,d}\,=\,e^{-dx_0}\,\left(\frac{1}{48 d}\,+\,\sum_{ k=0}^{N-1}\,a_k^1\,d^{-k-3/2}\,+\,O(d^{-N-3/2})\right).$$
\end{thm}

\begin{cor}\label{cor:2}
In case of $\CC P^2$, we have
\begin{equation}\label{eq:gw-4-2}
\lim_{d\to \infty}\,\sqrt[d]{n_{1,d}}\,=\,
\lim_{d\to \infty}\,\sqrt[d]{n_{0,d}}\,=\,e^{-x_0}.
\end{equation}
\end{cor}
This affirms a conjecture in \cite{zinger11}. It is plausible that $\sqrt[d]{n_{g,d}}$ converge to a fixed number independent of $g$. In general, we expect the same, i.e., the $\ell$-th root of $\psi_{g,\ell a}(0,0,\cdots,0,k_m(\ell),\cdots,k_L(\ell))$ converges to a fixed number which is related to the convergent radius of the generating functions.

It remains to understand geometric information encoded in the coefficients $a_k^{0}$ and $a_k^1$ in the above expansions.

\section{Consequences of the WDVV equation}

In this section, we show some consequences of the WDVV equation.
First we recall for any symplectic manifold $M$, $\bF_0^M$ satisfies the WDVV equation:
\begin{equation}\label{eq:gw-6}
\sum_{b,c}\,\frac{\partial ^3 \bF_0^M}{\partial t_i\partial j \partial b}\,\eta^{bc}\,\frac{\partial ^3 \bF_0^M}{\partial t_k\partial l \partial c}
\,=\,\sum_{b,c}\,\frac{\partial ^3 \bF_0^M}{\partial t_i\partial l \partial b}\,\eta^{bc}\,\frac{\partial ^3 \bF_0^M}{\partial t_k\partial j \partial c},
\end{equation}
where $i,j,k,l$ run over $0,\cdots, L$.

If we let $m\in [2, L]$ such that $d_1=1$ for $i < m$ and $d_i > 1$ for $i \ge m$, then
\begin{equation}\label{eq:gw-5}
\bF_0^M(t_0,\cdots,t_L)\,=\,p\,+\,\sum_{a\not= 0}\,\psi_{0,a}(0,\cdots,0,k_{m},\cdots,k_L) \,e^{\sum_{i=1}^{m-1}\,t_i\,\gamma_i(a)},
\end{equation}
where $p$ is the homogeneous cubic polynomial in $t_0,\cdots,t_L$ which gives rise to the cup product on $H^*_{ev}(M)$. Also $F_0$ satisfies the WDVV equation:

Now we assume $M=\CC P^2$, then $H_2(M,\ZZ)=\ZZ $ and \eqref{eq:gw-5} becomes
\begin{equation}\label{eq:gw-7}
\bF^{\CC P^2}_0(t)\,=\,\frac{1}{2}(t_2t_0^2+t_0t_1^2)+\sum_{k=1}^{\infty}\,n_{0,k}\,t_2^{3k-1}\,e^{kt_1},
\end{equation}
where $t=(t_0,t_1,t_2)$. We can further write
\begin{equation}\label{eq:gw-8}
\bF^{\CC P^2}_0\,=\,\frac{1}{2}(t_2t_0^2+t_0t_1^2)\,+\,t_2^{-1}F_0(t_1+3\ln t_2),
\end{equation}
where
$$F_0(z)\,=\,\sum_{d=1}^\infty\, n_{0,d} \,e^{d z}.$$

Using the WDVV equation which $\bF=\bF_0^{\CC P^2}$, we have
$$\bF_{112}^2\,=\,\bF_{111}\bF_{122}+\bF_{222},\ \ \bF_{ijk}\,=\,\frac{\partial^3\bF}{\partial t_i\partial t_j\partial t_k}.$$
Direct computations show:
$$\bF_{111}=t_2^{-1}F_0'''(t_1+3\ln t_2),\ \ \ \bF_{112}=t_2^{-1}(3F_0'''-F_0'')(t_1+3\ln t_2),$$
$$\bF_{122}=t_2^{-1}(9F_0'''-9F_0''+2F_0')(t_1+3\ln t_2)$$
and
$$ \bF_{222}=t_2^{-1}(27F_0'''-54F_0''+33F_0'-6F_0)(t_1+3\ln t_2).$$
It follows
$$(3F_0'''-F_0'')^2=F_0'''(9F_0'''-9F_0''+2F_0')+27F_0'''-54F_0''+33F_0'-6F_0.$$
This can be simplified as
\begin{equation}\label{1}(27+2F_0'-3F_0'')F_0'''=6F_0-33F_0'+54F_0''+(F_0'')^2,\end{equation}

The following lemma is taken from \cite{zinger11}.
\begin{lem}\label{lem:1}
We have the estimates for $n_{0,d}$:
\begin{equation}\label{eq:gw-5-1}
\left(\frac{1}{27}\right)^d\,d^{-\frac{7}{2}} \,\le\,n_{0,d}\,\le\,3\,\left(\frac{4}{15}\right)^d\,d^{-\frac{7}{2}}.
\end{equation}
\end{lem}
\begin{proof}
We outline its proof by following \cite{zinger11}. First we observe: If $n_d$ is a sequence of numbers satisfying:
$$n_d\,=\,\sum_{d_1,d_2\ge 0, d_1+d_2=d}\,\frac{f(d_1) f(d_2)}{f(d)}\,n_{d_1} n_{d_2}$$
for some $f:\ZZ\mapsto \RR^+$, then
$$n_d\,=\,\frac{(2d-2)!}{d!(d-1)!}\,(f(1)\,n_1)^d,~~~d\ge 1.$$
It is proved in \cite{ruantian1} by using the WDVV equation
$$n_{0,d}\,= \,\sum_{d_1,d_2\ge 0, d_1+d_2=d} \,T(d_1,d_2)\,n_{0,d_1} n_{0,d_2},$$
where
$$T(d_1,d_2)\,=\,\frac{d_1d_2(3d_1d_2(d+2) - 2d^2)}{2(3d-3)(3d-2)(3d-1)}.$$
It is easy to see
$$\frac{f_1(d_1)f_1(d_2)}{f_1(d)}\,\le T(d_1,d_2)\,\frac{f_2(d_1)f_2(d_2)}{f_2(d)},$$
where $f_1(d) = d(3d-2)/54$ and $f_2(d)= 2 d^2/15$.
If we denote by $n_d^j$ the sequence of numbers determined by the recursive formula for above $n_d$'s together with $f$ replaced $f_j$ and $n_1^j=1$, then
by induction, we can show $n_d^1\le n_{0,d}\le n_d^2$. Therefore, we have
$$\frac{(2d-2)!}{d!(d-1)!}\,\left(\frac{1}{54}\right)^d\,\le n_{0,d}\,\le\, \frac{(2d-2)!}{d!(d-1)!}\,\left(\frac{2}{15}\right)^d.$$
The lemma follows from this and Stirling formula for $d!$.
\end{proof}

\section{Proof of Theorem \ref{th:main}}

First we observe that as a consequence of Lemma \ref{lem:1}, we get a convergent radius $x_0$ such that the power series
$\sum_{d=1}^\infty \,n_{0,d}\,e^{d z}$ converges for any $z\in \CC$ with $Re(z) < x_0$, so $F_0(z)$ is well-defined in the region $\{z\,|\,Re(z) < x_0\}$.

Since $n_{0,d} \ge 1$, as observed in \cite{zinger11}, we can deduce
\begin{equation}\label{eq:gw-5-0}
0\,<\,F_0(z)\,<\,F_0'(z)\,<\,F_0''(z)\,<\,F_0'''(z),~~~~~~\forall\ z\in (-\infty,x_0)
\end{equation}
and
\begin{equation}\label{eq:gw-5-2}27+2F_0'-3F_0''\,>\,0,~~~~~~ \forall\ z\in (-\infty,x_0).
\end{equation}
By \eqref{eq:gw-5-0}, the series for $F_0$, $F_0'$ and $F_0''$ increase along $(-\infty,x_0)$. By \eqref{eq:gw-5-2}, $F_0< F_0' < F_0'' < 27$ on $(-\infty,x_0)$.
So $F_0$, $F_0'$ and $F_0''$ converge at $z=x_0$. Moreover, $27+2F_0'-3F_0''=0 $ at $z=x_0$, otherwise, \eqref{1} could be used to compute all derivatives of $F_0$ at $z=x_0$ and get a contradiction to the fact that $x_0$ is the convergent radius.

Clearly, $F_0$ is analytic in $\{\text{Re}\ z<x_0\}$.
\begin{lem}
\label{lem:2}
$F_0$ can be analytically continued to $\{\text{Re}\ z<x_0+\delta_0,\ 0\leq\text{Im}\ z\leq 2\pi\}, $ for some $0<\delta_0<1,\ i.e., $ $F_0$ is analytic in $\{\text{Re}\ z<x_0+\delta_0,\ 0<\text{Im}\ z< 2\pi\}, $ and continuous in $\{\text{Re}\ z<x_0+\delta_0,\ 0\leq\text{Im}\ z\leq 2\pi\}$, and $F_0$ has the expansion at $x_0$,
\begin{equation}\label{2}F_0(x_0+re^{i\theta})\,=\,\sum_{d=0}^{\infty}\,a_{d}\,r^{d/2}e^{i\theta d/2},\ \ \forall\ r<\delta_0,\ 0\leq\theta\leq\pi,\end{equation}
\begin{equation}\label{3}F_0(x_0+2\pi i +re^{i\theta})\,=\,\sum_{d=0}^{\infty}\,a_{d}\,r^{d/2}e^{i\theta d/2},\ \ \forall\ r<\delta_0,\ \pi\leq\theta\leq{2}\pi,
\end{equation}
where
$a_1\,=\,a_3=0$. Also we have $ |a_d|\,\delta_0^{d/2}\leq C$.
\end{lem}

\begin{proof}
First we observe
$$3F_0''-2F_0'\,=\,\sum\limits_{d=1}^{\infty}\,d (3d-2)\,n_{0,d}\,e^{dz}.$$
Since $n_{0,d}>0$, we have for any $t\in (0, 2\pi)$
$$|(3F_0''-2F_0')(x_0+it)|\,<\,(3F_0''-2F_0')(x_0)\,=\,27.$$
So $(27-3F_0''+2F_0')(x_0+it)\,\neq\,0$. Then, using \eqref{1}, one can easily show that
$F_0$ can be analytically continued over a neighborhood of each point $x_0+it$ with $t\in (0,2\pi)$. Hence, it suffices to prove the analytic continuation
of $F_0$ around $x_0$ and $x_0+ 2\pi i$.

Introduce new variables $t,x,y,w$ such that
$$\frac{dt}{dz}=\frac{1}{27+2F_0'-3F_0''},\ \ x\,=\,9F_0''-9F_0'+2F_0,\ \ y\,=\,3F_0''-F_0',\ \ w\,=\,F_0'',$$
for example, we can take
\begin{eqnarray}
t(z)&=&\frac{z}{27}\,-\,\int_0^{+\infty}\left(\frac{1}{27+2F_0'(z-s)-3F_0''(z-s)}-\frac{1}{27}\right)ds\nonumber\\
x(z)&=&\sum_{d=1}^{\infty}\,(3d-1)(3d-2)\,n_{0,d}\,e^{dz}\nonumber\\
y(z)&=&\sum_{d=1}^{\infty}\,d (3d-1)\,n_{0,d}\,e^{dz}\nonumber\\
w(z)&=&\sum_{d=1}^{\infty}\,d^2\,n_{0,d}\,e^{dz}.\nonumber
\end{eqnarray}
In fact,
$$(\bF_{22},\bF_{12},\bF_{11})\,=(0,0,t_0)+\,t_2^{-1}\,(x,y,w)(t_1+3\ln t_2).$$
Then $t\in \mathbb{R},\ x,y,w>0,$ and $x,y,w,t$ are strictly increasing on $(-\infty,x_0)$. Clearly, $x,y,w\to 0,\ t\to-\infty$ for $ z\to-\infty$, in particular,
$t(z)$ has an inverse $z(t)$ which maps $(-\infty, t_0')$ onto $(-\infty,x_0)$, where $t_0'=\lim_{z\to x_0} t(z)$. Note that $t_0' < \infty$
since $w$ blows up at finite time by the last equation in \eqref{4-1}.
Writing (\ref{1}) in terms of the variable $t$, we get
\begin{equation}\label{4-1}\frac{d}{dt}\left(\ba{l}x\\y\\w\\z\ea\right)
=\left(\ba{l}27x+4y^2\\9x+18y+2yw\\3x+6y+9w+w^2\\27-2y+3w\ea\right).
\end{equation}
Then \eqref{4-1} has a real analytic solution $ (\widehat{x},\widehat{y},\widehat{w},\widehat{z})(t)$ in the maximal interval $(-\infty,t_0)$ for some $t_0\in \mathbb{R}$ which extends the solution $(x,y,w,z)$ on $(-\infty,t_0')$, that is,
$$(x(z),y(z),w(z),z)=(\widehat{x},\widehat{y},\widehat{w},\widehat{z})(t(z)),~~~\forall z\in (-\infty,x_0).$$
Using \eqref{4-1}, we deduce that $\widehat{z}\in \mathbb{R},\ \widehat{x},\widehat{y},\widehat{w}$ are
strictly increasing positive functions on $(-\infty,t_0)$, and
$$\lim_{t\to t_0-}\widehat{x}(t)+\widehat{y}(t)+\widehat{w}(t)=+\infty.$$
Moreover, it follows from the equations on $\widehat{x},\widehat{y},\widehat{w}$ in \eqref{4-1}
$$\frac{d (\widehat{x}+\widehat{y}+\widehat{w})}{dt} \,\le\, (39+4\widehat{y}+\widehat{w})(\widehat{x}+\widehat{y}+\widehat{w}).$$
It implies
$$\frac{d \log (\widehat{x}+\widehat{y}+\widehat{w})}{dt} \,\le\, 4 (10 + \widehat{y}+\widehat{w}).$$
Hence, we have
\begin{equation}\label{4-2}
\int_{t_0-1}^{t_0}(\widehat{y}(t)+\widehat{w}(t))dt\,=\,+ \infty.
\end{equation}

Observe that
$$2y-3w\,=\,3F_0''-2F_0'\,=\,\sum_{d=1}^{\infty}\,d (3d-2)\,n_{0,d}\,e^{dz}\,>\,0$$
and $2y-3w$ is strictly increasing on $(-\infty,x_0)$. It follows that $(2\widehat{y}-3\widehat{w})(t)>0$ for $t=t(z),\ z\in(-\infty,x_0)$ and $(2\widehat{y}-3\widehat{w})(t)\to 0$ for $t\to-\infty$.
Using \eqref{4-1}, we have
\begin{equation}
\label{4-3}
(2\widehat{y}-3\widehat{w})'\,=\,9\widehat{x}\,+\,(9+\widehat{w})(2\widehat{y}-3\widehat{w})\,+\,2 \widehat{w}\widehat{y}.
\end{equation}
Therefore, $(2\widehat{y}-3\widehat{w})(t)>0$ and $(2\widehat{y}-3\widehat{w})'(t)>0$ for $t\in(-\infty,t_0).$
Further, we derive from \eqref{4-3} the following differential inequality on $(t_0-1, t_0)$,
$$(2\widehat{y}-3\widehat{w})'(t)\,\ge\,(2\widehat{y}-3\widehat{w})(t_0-1)\widehat{w}(t)\,+\,2\widehat{w}(t_0-1)\widehat{y}(t).$$
Combining this with \eqref{4-2}, we get
$$\lim_{t\to t_0-}(2\widehat{y}-3\widehat{w})(t)\,=\,+\infty,$$
so there exists a unique $t_1\in(-\infty,t_0) $ such that $ (2\widehat{y}-3\widehat{c})(t_1)=27$.
By \eqref{4-1}, this implies that $\widehat{z}'(t_1)=0$ and $\widehat{z}''(t_1)<0$.
Since $27+2F_0'-3F_0''=0 $ at $z=x_0$, we have $\widehat{z}'(t(x_0))=(27+2F_0'-3F_0'')(x_0)=0$, so $t_1=t(x_0)=t_0',\ x_0=\widehat{z}(t_1)$ and
there is a $\delta_1 >0$ such that for $|t|<\delta_1$,
$$\widehat{z}(t_1+t)\,=\,x_0\,+\,\sum_{k=2}^{\infty}\,b_k\,t^k,\ \ \widehat{w}(t_1+t)\,=\,\sum_{k=0}^{\infty}\,c_k\,t^k,$$
where $b_2<0,\ b_k,c_k\in \mathbb{R},\ c_1>0.$ As $ \widehat{w}(t)=w(\widehat{z}(t))$ for $t<t_1$, there is a $\delta_2 > 0$ such that for $0<z<\delta_2$,
$$w(x_0-z)\,=\,\sum_{k=0}^{\infty}\,c_k'\,z^{\frac{k}{2}}, $$
where $ c_k'\in \mathbb{R}, c_1'<0$. As $w(z)=F_0''(z)$, we have $$F_0(x_0-z)\,=\,F_0(x_0)\,-\,F_0'(x_0)\,z\,+\,\sum_{k=0}^{\infty}\,\frac{4\,c_k'}{(k+2)(k+4)}\,z^{\frac{k+4}{2}}.$$
Therefore (\ref{2}) is true for $\theta=\pi$ with
$$a_0\,=\,F_0(x_0),\,\, a_2\,=\,F_0'(x_0),\,\, a_1\,=\,a_3\,=\,0,\, \,a_k\,=\,\frac{4\,i^{-{k}}\,c_{k-4}'}{k(k-2)},~~\forall k\geq 4.$$
Since $F_0$ is $2\pi i$ periodic in $\{\text{Re}\ z<x_0\}$, \eqref{3} is also true for $\theta=\pi$. Thus we have the real analytic expressions
in \eqref{2} and \eqref{3} for $\theta=\pi$. It implies the analytic continuation
of $F_0$ around $x_0$ and $x_0+ 2\pi i$, and consequently, (\ref{2}) and
(\ref{3}).
\end{proof}

\begin{rem}
The analytic continuation of $F_0$ in \eqref{2} and \eqref{3} is a more precise form of the following expansion:
$$F_0(x_0+z)\,=\,\sum_{d=0}^{\infty}\,a_{d}\,z^{d/2},~~\forall\, |z|<\delta_0,~~a_1\,=\,a_3\,=\,0.$$
This expansion around $x_0$ was claimed in \cite{francesco} without a proof.
A justification was provided in \cite{zinger11}. Our proof above is different and clearer. Also, Lemma \ref{lem:2} is stronger than the expansion claimed in \cite{francesco}. It states that $F_0$ can be analytically continued in a neighborhood of $\{x_0+ it\,|\,t\in [0,2\pi]\}$. We need this stronger version in the subsequent arguments.
\end{rem}

Now we complete the proof of Theorem \ref{th:main}. Fix any $\delta\in (0,\delta_0)$, using contour integration,
we have
\begin{align}\label{4} n_{0,d}\,=&\,\frac{1}{2\pi i}\,\int_{x_0}^{x_0+2\pi i}\,F_0(z)\,e^{-dz}\,dz\nonumber\\
=&\,\frac{1}{2\pi i}\,\int_{0}^{\delta}\,F_0(x_0+t)\,e^{-d(x_0+t)}\,dt\nonumber \\
& +\,\frac{1}{2\pi }\,\int_{0}^{2\pi}\,F_0(x_0+\delta+it)\,e^{-d(x_0+\delta+it)}\,dt
\nonumber\\
&-\,\frac{1}{2\pi i}\,\int_{0}^{\delta}\,F_0(x_0+2\pi i+t)\,e^{-d(x_0+2\pi i+t)}\,dt .
\end{align}
One can easily have the following estimate:
\begin{equation}\label{5}\left|\frac{1}{2\pi }\,\int_{0}^{2\pi}\,F_0(x_0+\delta+it)\,e^{-d(x_0+\delta+it)}\,dt\right|\,\leq\, C_1\,e^{-d(x_0+\delta)},\end{equation}
where $C_1=\max\limits_{0\leq t\leq 2\pi}|F_0(x_0+\delta+it)|< +\infty$.
It follows from \eqref{2}, \eqref{3} and the fact that $a_1\,=\,a_3=0$
\begin{align}\label{6}&\frac{1}{2\pi i}\,\int_{0}^{\delta}\,F_0(x_0+t)\,e^{-d(x_0+t)}\,dt\,-\,\frac{1}{2\pi i}\,\int_{0}^{\delta}\,F_0(x_0+2\pi i+t)\,e^{-d(x_0+2\pi i+t)}\,dt\nonumber \\=&\frac{1}{2\pi i}\,\int_{0}^{\delta}\,(F_0(x_0+t)\,-\,F_0(x_0+2\pi i+t))\,e^{-d(x_0+t)}\,dt\nonumber \\=&\frac{1}{2\pi i}\,\int_{0}^{\delta}\,\left(\sum_{k=0}^{\infty}\,a_{k}\,t^{k/2}\,-\,\sum_{k=0}^{\infty}\,a_{k}t^{k/2}(-1)^k\right)\,e^{-d(x_0+t)}dt\nonumber \\
=&\sum_{ k=3}^{\infty}\,\frac{a_{2k-1}}{\pi i}\,\int_{0}^{\delta}\,t^{k-1/2}\,e^{-d(x_0+t)}\,dt\,=\,e^{-dx_0}\,\sum_{ k=3}^{\infty}\,\frac{a_{2k-1}}{\pi i} \,A(k,\delta,d),\end{align}
where $$A(k,\delta,d)\,=\,\int_{0}^{\delta}\,t^{k-1/2}\,e^{-dt}\,dt\,=\,d^{-k-1/2}\,\int_{0}^{d \delta}\,t^{k-1/2}\,e^{-t}\,dt.$$
Clearly, we have
$$0\,<\,A(k,\delta,d)\,<\,d^{-k-1/2}\,\int_{0}^{+\infty}\,t^{k-1/2}\,e^{-t}\,dt\,=\,\frac{\Gamma(k+1/2)}{d^{k+1/2}}$$
and
$$0\,<\,A(k+1,\delta,d)\,<\,\delta\, A(k,\delta,d).$$
Fix any $N\in \mathbb Z$ and $N>3$, we have for $ 3\leq k<N,$
\begin{align*}0&<\,\Gamma(k+1/2)\,-\,d^{k+1/2}A(k,\delta,d)\,=\,\int_{d\delta}^{+\infty}\,t^{k-1/2}\,e^{-t}\,dt\\ &\leq\, (d\delta)^{k-N}\,\int_{d\delta}^{+\infty}\,t^{N-1/2}e^{-t}\,dt\,<\,(d\delta)^{k-N}\,\Gamma(N+1/2),\end{align*}
We further estimate
\begin{align}\label{7}&\left|\sum_{ k=3}^{\infty}\,\frac{a_{2k-1}}{\pi i}\, A(k,\delta,d)\,-\,\sum_{ k=3}^{N-1}\,\frac{a_{2k-1}\Gamma(k+1/2)}{\pi id^{k+1/2}}\right|\nonumber \\ \leq& \,\sum_{ k=N}^{\infty}\,\frac{|a_{2k-1}|}{\pi } \,A(k,\delta,d)\,+\,\sum_{ k=3}^{N-1}\,\frac{|a_{2k-1}|(\Gamma(k+1/2)-d^{k+1/2}A(k,\delta,d))}{\pi d^{k+1/2}}\nonumber \\ \leq& \sum_{ k=N}^{\infty}\frac{|a_{2k-1}|}{\pi } \delta ^{k-N}A(N,\delta,d)+\sum_{ k=3}^{N-1}\frac{|a_{2k-1}|(d\delta)^{k-N}\Gamma(N+1/2)}{\pi d^{k+1/2}}\nonumber \\ \leq& \sum_{ k=N}^{\infty}\frac{|a_{2k-1}|\Gamma(N+1/2)}{\pi d^{N+1/2}} \delta ^{k-N}+\sum_{ k=3}^{N-1}\frac{|a_{2k-1}|\delta^{k-N}\Gamma(N+1/2)}{\pi d^{N+1/2}}\nonumber \\=&\,\frac{\Gamma(N+1/2)}{\pi d^{N+1/2}}\,\sum_{ k=3}^{\infty}\,|a_{2k-1}|\,\delta ^{k-N}\,=\,C_2\,\frac{\Gamma(N+1/2)}{\pi (d\delta)^{N+1/2}},\end{align}
where $C_2\, =\, \sum \limits_{ k=3}^{\infty}\,|a_{2k-1}|\,\delta ^{k+1/2}\,<\,+\infty$. It follows from \eqref{4}, \eqref{5}, \eqref{6} and \eqref{7} \begin{align*}\left|e^{dx_0}\,n_{0,d}\,-\,\sum_{ k=3}^{N-1}\,\frac{a_{2k-1}\Gamma(k+1/2)}{\pi id^{k+1/2}}\right|\,\leq\, C_1\,e^{-d\delta}\,+\,C_2\,\frac{\Gamma(N+{1}/{2})}{\pi (d\delta)^{N+\frac{1}{2}}}\,\leq\, \frac{C(N)}{(d\delta)^{N+\frac{1}{2}}}.\end{align*}
Hence, we can write
$$n_{0,d}\,=\,e^{-dx_0}\,\left(\sum_{ k=3}^{N-1}\,\frac{a_{2k-1}\Gamma(k+1/2)}{\pi id^{k+1/2}}\,+\,O(d^{-N-1/2})\right).$$
Since $n_{0,d}\in\RR$, so does $-i a_{2k-1}$.
Set
$$a_k^{0} \,=\,\frac{a_{2k-1}\Gamma(k+1/2)}{\pi i}.$$
In particular,
$$a_3^{0} \,=\,\dfrac{a_{5}\,\Gamma(7/2)}{\pi i}\,=\,\dfrac{4i^{-{5}}\,c_{1}'\,\Gamma(7/2)}{15\pi i}\,=\,-\dfrac{4c_{1}'\,\Gamma(7/2)}{15\pi}\,>\,0.$$
Thus we get the asymptotic expansion of $n_{0,d}$ as we stated in Theorem \ref{th:main}.

\section{Proof of Theorem \ref{th:main-2}}

We will adopt the notations from last section. First we define a generating function
\begin{equation}\label{eq:f1}
F_1(z)\,=\,\sum_{d=1}^{\infty}\,n_{1,d}\,e^{dz}.
\end{equation}
It follows from the Eguchi-Hori-Xiong recursion formula as shown in \cite{pande} that $F_1$ satisfies
\begin{equation}\label{8}
(27+2F_0'-3F_0'')\,F_1'\,=\,\frac{1}{8}\,(F_0'''-3F_0''+2F_0').
\end{equation}
By \eqref{4-1}, $ \widehat{w}'(t)>0$, so we have $c_1>0$, $c_1'>0$ and $a_5\neq 0$. Since $(27+2F_0'-3F_0'')(x_0)=0$ and $a_1=a_3=0$, by \eqref{2},  for any $z$ with $|z|<\delta_0$ and $0\leq \arg z<\frac{3}{2}\pi$, we have
$$(27+2F_0'-3F_0'')(x_0+z)\,=\,\sum_{d=1}^{\infty}\,\left(-\frac{3(d+4)(d+2)}{4}\,a_{d+4}+(d+2)\,a_{d+2}\right)\,z^{d/2}$$
and the expansion of $(F_0'''-3F_0''+2F_0')(x_0+z)$ is
$$\sum_{d=-1}^{\infty}\,\Big(\frac{(d+6)(d+4)(d+2)}{8}\,a_{d+6}\\
-\,\frac{3(d+4)(d+2)}{4}\,a_{d+4}+(d+2)\,a_{d+2}\Big)\,z^{d/2}.
$$
Therefore, we can write
$$F_1'(x_0+z)\,=\, \sum_{d=-2}^{\infty}\,a'_d\,z^{d/2},$$
where $a'_{-2}\,=\,- \,\frac{1}{48}$.
Hence, $F_1'$ is analytic in the region $\{\text{Re}\ z<x_0\}$. Since $(27-3F_0''+2F_0')(x_0+it)\,\neq\,0$ for $0<t<2\pi$, we can deduce from \eqref{8} 
that $F_1'$ can be analytically continued to $\{\text{Re}\ z<x_0+\delta_0,\ 0\leq\text{Im}\ z\leq 2\pi,\ z\neq x_0,x_0+2\pi i\}$ for some $0<\delta_0<1$.
Moreover, we have the following expansions: \begin{equation}\label{9}
F_1'(x_0+re^{i\theta})\,=\,\sum_{d=-2}^{\infty}\,a'_{d}\,r^{d/2}\,e^{i\theta d/2},\ \ \forall\ 0<r<\delta_0,\ 0\leq\theta\leq\pi,
\end{equation}
\begin{equation}\label{10}F_1'(x_0+2\pi i +re^{i\theta})\,=\,\sum_{d=-2}^{\infty}\,a'_{d}\,r^{d/2}
\,e^{i\theta d/2},\ \ \forall\ 0<r<\delta_0,\ \pi\leq\theta\leq{2}\pi.
\end{equation}
Fix $0<\delta<\delta_0$, using contour integration, we have for all $0<\delta_1<\delta$,
\begin{align}\label{11}d n_{1,d}\,=&\,\frac{1}{2\pi i}\,\int_{x_0-\delta_1}^{x_0-\delta_1+2\pi i}\,F_1'(z)\,e^{-dz}dz\nonumber\\
=&\,-\frac{\delta_1}{2\pi }\,\int_{0}^{\pi}\,F_1'(x_0+\delta_1e^{it})\,e^{-d(x_0+\delta_1e^{it})}\,e^{it}\,dt\nonumber \\
&+\,\frac{1}{2\pi i}\,\int_{\delta_1}^{\delta}\,F_1'(x_0+t)e^{-d(x_0+t)}\,dt\nonumber\\
&+\,\frac{1}{2\pi }\,\int_{0}^{2\pi}\,F_1'(x_0+\delta+it)\,e^{-d(x_0+\delta+it)}\,dt
\nonumber \\&-\,\frac{1}{2\pi i}\,\int_{\delta_1}^{\delta}\,F_1'(x_0+2\pi i+t)\,e^{-d(x_0+t)}\,dt\nonumber\\
&- \,\frac{\delta_1}{2\pi }\,\int_{\pi}^{2\pi}\,F_1'(x_0+2\pi i+\delta_1e^{it})\,e^{-d(x_0+\delta_1e^{it})}\,e^{it}\,dt .
\end{align}
It is easy to see
\begin{equation}\label{12}\left|\frac{1}{2\pi }\,\int_{0}^{2\pi}\,F_1'(x_0+\delta+it)\,e^{-d(x_0+\delta+it)}\,dt\right|
\,\leq \,C_1'\,e^{-d(x_0+\delta)},\end{equation}
where $C_1'\,=\,\max\limits_{0\leq t\leq 2\pi} |F_1'(x_0+\delta+it)|<+\infty$.
By \eqref{9} and \eqref{10}, as $\delta_1\to 0$, we get
\begin{align}\label{13}&\frac{1}{2\pi i}\,\int_{\delta_1}^{\delta}\,F_1'(x_0+t)e^{-d(x_0+t)}\,dt\,-\,\frac{1}{2\pi i}\,\int_{\delta_1}^{\delta}\,F_1'(x_0+2\pi i+t)\,e^{-d(x_0+t)}\,dt\nonumber \\=&\,\frac{1}{2\pi i}\,\int_{\delta_1}^{\delta}\,(F_1'(x_0+t)-F_1'(x_0+2\pi i+t))\,e^{-d(x_0+t)}\,dt\nonumber \\=&\,\frac{1}{2\pi i}\,\int_{\delta_1}^{\delta}\,\left(\sum_{k=-2}^{\infty}a'_{k}t^{k/2}\,-\,\sum_{k=-2}^{\infty}\,a'_{k} \,t^{k/2}(-1)^k\right)\,e^{-d(x_0+t)}\,dt
\nonumber \\=&\,\sum_{ k=0}^{\infty}\,\frac{a'_{2k-1}}{\pi i}\int_{\delta_1}^{\delta}\,t^{k-1/2}\,e^{-d(x_0+t)}\,dt\to e^{-d x_0}\sum_{ k=0}^{\infty}\,\frac{a'_{2k-1}}{\pi i} \,A(k,\delta,d),
\end{align}
\begin{align}\label{11}&-\frac{\delta_1}{2\pi }\,\int_{0}^{\pi}\,F_1'(x_0+\delta_1e^{it})\,e^{-d(x_0+\delta_1e^{it})} e^{it}\,dt\nonumber\\
&-\,\frac{\delta_1}{2\pi }\,\int_{\pi}^{2\pi}\,F_1'(x_0+2\pi i+\delta_1e^{it})\,e^{-d(x_0+\delta_1e^{it})}\,e^{it}\,dt\to\nonumber \\ & -\frac{1}{2\pi }\,\int_{0}^{\pi}\,a'_{-2}\,e^{-it}\,e^{-dx_0}\,e^{it}\,dt\,-\,\frac{1}{2\pi }\,\int_{\pi}^{2\pi}\,a'_{-2}\,e^{-it}\,e^{-dx_0}\,e^{it}\,dt\,=\,- a'_{-2}\,e^{-dx_0}.\nonumber
\end{align}
Here we have used the dominated convergence theorem. Then, arguing in a similar way as we did in the case of $n_{0,d}$, we can deduce $$d n_{1,d}\,=\,e^{-dx_0}\left(-a'_{-2}\,+\,\sum_{ k=0}^{N-1}\,\frac{a'_{2k-1}\Gamma(k+1/2)}{\pi i\,d^{k+1/2}}\,+\,O(d^{-N-1/2})\right).$$
Since $a'_{-2}\,=\,-\,\dfrac{1}{48}$, we have
$$n_{1,d}\,=\,e^{-dx_0}\left(\frac{1}{48d}\,+\,\sum_{ k=0}^{N-1}\,\frac{a'_{2k-1}\Gamma(k+1/2)}{\pi id^{k+3/2}}\,+\,O(d^{-N-3/2})\right).$$
Since $n_{1,d}\in\RR$, so does $-i a_{2k-1}'$.
Set
$$a_k^{1} \,=\,\frac{a_{2k-1}'\Gamma(k+1/2)}{\pi i}.$$
Thus we get the asymptotic expansion of $n_{1,d}$ as we stated in Theorem \ref{th:main-2}.

\end{document}